\documentclass[aap]{aptpub} 
\usepackage{amsmath,graphicx,natbib, amssymb,epsfig,color} 
\authornames{E. Damek, T. Mikosch, J. Rosi\'nski, G. Samorodnitsky} 
\shorttitle{Inverse problems for regular variation} 
\usepackage{color} 
\usepackage{amssymb,amsfonts,amsmath} 
\usepackage{enumerate} 
 
\newcommand{\red}{\color{black}} 
 
\definecolor{darkblue}{rgb}{.1, 0.1,.8} 
\definecolor{darkgreen}{rgb}{0,0.8,0.2} 
\definecolor{darkred}{rgb}{.8, .1,.1}

\newcommand{\id }{infinitely divisible} 
 
\newcommand{\ts}{time series} 
\newcommand{\tsa}{\ts\ analysis}

\newcommand{\RV}{{\rm RV}}

\newcommand{\bfz}{{\bf z}} 
\newcommand{\bth}{\begin{theorem}} 
\newcommand{\ethe}{\end{theorem}} 
 
\newcommand{\bre}{\begin{remark}\em } 
\newcommand{\ere}{\end{remark}} 
 
\newcommand{\ble}{\begin{lemma}} 
\newcommand{\ele}{\end{lemma}}

\newcommand{\bde}{\begin{definition}} 
\newcommand{\ede}{\end{definition}} 
\newcommand{\bco}{\begin{corollary}} 
\newcommand{\eco}{\end{corollary}} 
 
\newcommand{\bpr}{\begin{proposition}} 
\newcommand{\epr}{\end{proposition}} 
 
\newcommand{\bexer}{\begin{exercise}} 
\newcommand{\eexer}{\end{exercise}} 
 
\newcommand{\bexam}{\begin{example}\rm } 
\newcommand{\eexam}{\end{example}}

\newcommand{\btab}{\begin{tab}} 
\newcommand{\etab}{\end{tab}} 

\newcommand{\rv}{random variable}

\newcommand{\sign}{{\rm sign}}

\newcommand{\rhs}{right-hand side}

\newcommand{\beao}{\begin{eqnarray*}} 
\newcommand{\eeao}{\end{eqnarray*}\noindent} 
 
\newcommand{\beam}{\begin{eqnarray}} 
\newcommand{\eeam}{\end{eqnarray}\noindent} 
 
\newcommand{\beqq}{\begin{equation}} 
\newcommand{\eeqq}{\end{equation}\noindent} 
 
\newcommand{\bce}{\begin{center}} 
\newcommand{\ece}{\end{center}} 
 
\newcommand{\barr}{\begin{array}} 
\newcommand{\earr}{\end{array}}

\newcommand{\stv}{\stackrel{v}{\rightarrow}}

\newcommand{\eqd}{\stackrel{d}{=}}

\newcommand{\vague}{\stackrel{\lower0.2ex\hbox{$\scriptscriptstyle 
                    \it{v} $}}{\rightarrow}} 
\newcommand{\weak}{\stackrel{\lower0.2ex\hbox{$\scriptscriptstyle 
                    \it{w} $}}{\rightarrow}} 
\newcommand{\what}{\stackrel{\lower0.2ex\hbox{$\scriptscriptstyle 
                    \it{\hat{w}} $}}{\rightarrow}} 
 
\newcommand{\bdis}{\begin{displaymath}} 
\newcommand{\edis}{\end{displaymath}\noindent}

\newcommand{\R}{\mathbb{R}} 
 
\newcommand{\xbold}{{\mathbf{x}}} 
\newcommand{\bx}{\xbold} 
 
\newcommand{\ybold}{{\mathbf{y}}} 
\newcommand{\by}{\ybold} 
\newcommand{\zbold}{{\mathbf{z}}} 
\newcommand{\bz}{\zbold}

\newcommand{\nto}{n\to\infty}

\newcommand{\ov}{\overline} 
 
\newcommand{\wh}{\widehat} 
\newcommand{\vep}{\varepsilon}

\newcommand{\regvary}{regularly varying} 
 
\newcommand{\regvar}{regular variation} 

\newcommand{\bbr}{{\mathbb R}}

\newcommand{\bbz}{{\mathbb Z}} 
 
\newcommand{\bbd}{{\mathbb D}} 
 
\newcommand{\bbs}{{\mathbb S}} 
 
\newcommand{\bbq}{{\mathcal Q}}

\newcommand{\con}{convergence}

\newcommand{\st}{such that} 
\newcommand{\fif}{if and only if}

\newcommand{\ds}{distribution}

\newcommand{\seq}{sequence}

\newcommand{\ms}{measure}

\newcommand{\bfx}{{\bf x}} 
\newcommand{\bfX}{{\bf X}} 
 
\newcommand{\bfY}{{\bf Y}} 
\newcommand{\bfy}{{\bf y}} 
\newcommand{\bfA}{{\bf A}}

\newcommand{\bfZ}{{\bf Z}} 
 
\newcommand{\bfa}{{\bf a}} 
 
\newcommand{\bfd}{{\bf d}} 
\newcommand{\bfe}{{\bf 1}}

\newcommand{\bfs}{{\bf s}}

\allowdisplaybreaks 
\begin{document}

\title{General inverse problems for 
  regular variation} 
 
\authorone[University of Wroclaw]{Ewa Damek} 
\authortwo[University of Copenhagen]{Thomas Mikosch} 
\authorthree[University of Tennessee]{Jan Rosi\'nski} 
\authorfour[Cornell University]{Gennady Samorodnitsky} 
\addressone{Mathematical Institute,  
Wroclaw University, 
Pl. Grunwaldzki 2/4, 
50-384 Wroclaw, 
Poland, Ewa.Damek@math.uni.worc.pl} 
\addresstwo{Department of Mathematics, 
University of Copenhagen, 
Universitetsparken~5, 
DK-2100 Copenhagen, 
Denmark, mikosch@math.ku.dk} 
\addressthree{Department of Mathematics, 
227 Ayres Hall, 
University of Tennessee, 
Knoxville, TN 37996-1320, 
U.S.A., rosinski@math.utk.edu} 
\addressfour{School of Operations Research and Information Engineering, 
and Department of Statistics, 
Cornell University, 
220 Rhodes Hall, 
Ithaca, NY 14853,  
U.S.A., gs18@cornell.edu}


\begin{abstract} 
Regular variation of distributional tails is known to be preserved by 
various linear transformations of some random structures.  
An inverse problem for regular 
variation aims at understanding whether the regular variation of a 
transformed random object is caused by regular variation of components 
of the original random structure. In this paper we build up on previous 
work and derive results in the multivariate case and in  
situations where regular variation is 
not restricted to one particular direction or quadrant.  
\end{abstract} 
\section{Introduction}\label{sec:1} 
\setcounter{equation}{0} 
The four 
authors of this paper are very much honored to contribute  
to this special issue of one of the oldest journals in applied 
probability. We wish happy birthday and a very long life to this 
excellent journal.   
Two of us, Thomas Mikosch and Gennady Samorodnitsky,  
were inivited to contribute short papers   
to this special issue. With the permission of the editors, we merged 
efforts leading to this longer and more substantial paper.

 In this paper we study regular variation of the tails of measures  
on $\bbr^d$, most 
importantly probability measures. Stated somewhat 
vaguely, it is well known that regular variation tends to be 
preserved by various linear operations (such as linear transformations of the 
space, convolutions, integrals, etc.) We would like to understand to 
what degree the inverse statements are valid. That is, if the result 
of a linear operation on a measure is regularly varying in the 
appropriate space, was the original measure necessarily regularly 
varying as well? 
 
This type of questions is often referred to as {\em inverse problems 
for regular variation,} and in the previous paper 
\cite{jacobsen:mikosch:rosinski:samorodnitsky:2008}  a fairly complete 
answer to this problem for certain non-negative linear transformations 
of one-dimensional measures was given.  Our aims in this paper are  
to treat the inverse problem in the multivariate case and to get rid 
of the non-negativity assumption on the linear transformations.  
We are fairly successful in our latter task,  but only partially in
the former one.   
 
Now we will be more precise. Let $\R^d_0 = \R^d \backslash \{\bf0\}$  and  
$\ov\bbr_0^d =\ov\bbr^d \backslash \{\bf0\}$, where $\ov\bbr =[-\infty, \infty]$. Recall that a random vector $\bfX$ with values in $\bbr^d$ is said to be {\em \regvary } if there exists a non-null Radon  \ms\ $\mu_\bfX$ on 
the Borel $\sigma$-field of $\ov\bbr_0^d$ (that does not charge the set of infinite points) \st  
\beao 
\dfrac{P(s^{-1} \bfX\in \cdot )}{P(\|\bfX\|>s)}\stv \mu_{\bfX}\,,\quad 
\mbox{ as $s\to\infty$,} 
\eeao 
where $\stv$ stands for vague \con\ on the Borel $\sigma$-field of  
$\ov\bbr_0^d$; see 
e.g.   \cite{kallenberg:1983} or  
\cite{resnick:1987}. Recall that, in this context, a Borel measure is 
Radon if it is finite outside of any ball of positive radius 
centered at the origin.   
The \ms\ $\mu_{\bfX}$ necessarily satisfies the relation 
$\mu_{\bfX}(tA)=t^{-\alpha}\mu_{\bfX}(A)$, $t>0$, for all Borel sets 
$A$, some 
$\alpha> 0$. We will refer to $\alpha$ as the {\em index of \regvar } and 
$\mu_\bfX$ as the {\em tail \ms.} The notion of regular variation applies 
equally well to $\sigma$-finite Borel measures on $\bbr^d$ that are  
finite outside of any ball of a positive radius 
centered at the origin. Specifically, any such measure $\nu$  
is said to 
be {\em \regvary } if, as above, there is a non-null Radon  \ms\ $\mu$ 
on $\ov\bbr_0^d$ that does not charge the set of infinite 
points \st 
\beao 
\dfrac{\nu( s\cdot)}{ \nu(\{ \bfy:\, \|\bfy\|>s\})}\stv \mu\,,\quad 
\mbox{ as $s\to\infty$.} 
\eeao 
As in the case of probability measures, the limiting measure $\mu$ 
scales with index $\alpha>0$. We will write 
$\nu \in {\rm RV}(\alpha,\mu)$. Of course, this language allows the 
measure  
$\nu$ to be the law of a random vector $\bfX$, but in the case 
of random vectors it is even more common to simply write  
$\bfX\in {\rm RV}(\alpha,\mu_\bfX)$. 
\par 
To give a taste of linear operations on regularly 
varying measures we have in mind, we proceed with examples. The reader 
will notice that these examples are more general versions of the 
corresponding examples in 
\cite{jacobsen:mikosch:rosinski:samorodnitsky:2008}.  
 
\begin{example} \label{ex:linear.process} {\rm  ({\it Weighted sums}).  
\ Let $\Psi_j$, $j=1,2,\ldots$ be (non-random) $d\times m$ matrices 
and $(\bfZ^{(j)})$ an iid \seq\ of \regvary\ $\bbr^m$-valued random 
(column) vectors with a generic element $\bfZ\in{\rm 
  RV}(\alpha,\mu_{\bfZ})$.  
Then under appropriate size conditions on the matrices $(\Psi_j)$, the 
series $\bfX=\sum_{j=1}^\infty \Psi_j\,\bfZ^{(j)}$ converges with 
probability 1, and $\bfX$ is \regvary\ 
with index $\alpha$ and 
\beam\label{eq:10} 
\dfrac{P(s^{-1}\bfX \in \cdot)}{P(\|\bfZ\|>s)}\stv \sum_{j=1}^\infty 
\mu_{\bfZ}\circ \Psi_j^{-1}\,, \quad 
\mbox{ as $s\to\infty$}, 
\eeam 
assuming that the \rhs\ does not vanish; see  
\cite{hult:samorodnitsky:2008,hult:samorodnitsky:2010}. This statement is always true if the 
sum is finite; see  
\cite{resnick:willekens:1991},  \cite{basrak:davis:mikosch:2002a}.  
 
Is the converse statement true? That is, if $\bfX$ is \regvary , does 
it follow that the iid vectors $\bfZ_i$ are regularly varying as well? 
In  \cite{jacobsen:mikosch:rosinski:samorodnitsky:2008} 
this problem was solved   
for iid {\it real-valued} $Z_i$ and {\em non-negative} scalars $\Psi_j=\psi_j$. 
(Here and in what follows, we use the symbol $\psi_j$ for scalars 
instead of genuine matrices $\Psi_j$.) It was shown that  (under 
appropriate size conditions in the case of an infinite sum),  
$Z_i$ inherits 
\regvar\ with index $\alpha$ from $X$ if the condition 
\beam\label{eq:1} 
\sum_{j=1}^{{\red \infty}} \psi_j^{\alpha+i\theta} \ne 0\,, \quad \mbox{for all} \quad  \theta\in \bbr\,, 
\eeam 
holds. Moreover, if \eqref{eq:1} fails, then one can find iid $(Z_i)$ 
which are not \regvary\ but $X$ is \regvary . In this paper we want 
to extend the result to the multivariate case and/or  
drop the assumption of non-negative coefficients.  
}  
\end{example} 
 
\begin{example} \label{ex:products} {\rm  ({\it Products}).  
\ Let $\bfZ\in{\rm 
  RV}(\alpha,\mu_{\bfZ})$ be a  random (column) vector in $\bbr^m$ and   
$\bfA$ be a random $d\times m$ matrix, independent of $\bfZ$ such 
that its matrix norm satisfies $E\| \bfA\|^{\alpha+\vep}<\infty$ for 
some $\vep>0$. Then $\bfX = \bfA\bfZ$ is regularly varying with index 
$\alpha$ in $\bbr^d$, and 
\beam\label{e:product.regvar} 
\dfrac{P(s^{-1}\bfX \in \cdot)}{P(\|\bfZ\|>s)}\stv E\Bigl[ 
\mu_{\bfZ}\circ \bfA^{-1}\Bigr]\, \quad 
\mbox{ as $s\to\infty$,} 
\eeam 
provided the measure on the right-hand side does not 
vanish; see \cite{basrak:davis:mikosch:2002b}. Once again, is the 
converse statement true? That is, if $\bfX$ is \regvary , does 
it follow that the vector $\bfZ$  is \regvary\ (assuming that the random matrix  
$\bfA$ is suitably small)? In \cite{jacobsen:mikosch:rosinski:samorodnitsky:2008} it was shown 
that, if $A$ and $Z$ are real-valued and $A>0$, then $Z$ inherits 
\regvar\ with index $\alpha$ from $X$ if and only if  
\beam\label{e:cond.prod} 
E\bigl[ A^{\alpha+i\theta}\bigr] \ne 0\,, \quad  \mbox{for all} \quad  \theta\in \bbr\,. 
\eeam 
We would like to remove the restriction to one dimension and 
the assumption of non-negativity.  
} 
\end{example}{\red 
As in the one-dimensional non-negative case, these questions turn 
out to be related to a certain cancellation property of measures, 
which we address in Section \ref{sec:cancellation}. The proof of 
the cancellation property requires some abstract Fourier analytic
arguments. The reader interested in applications of these results in
the spirit of Examples~\ref{ex:linear.process} and \ref{ex:products}
is referred to Section~\ref{sec:sums}--\ref{sec:non-diagonal}. In
Section~\ref{sec:sums} we study the inverse problem for
weighted sums of a multivariate iid \seq .
In Section~\ref{sec:products} we consider the corresponding problem for
matrix products, where the random matrix has diagonal structure. Some
examples in the case of non-diagonal deterministic matrices are given
in Section~\ref{sec:non-diagonal}. While the results in
Section~\ref{sec:sums} yield a rather complete picture for weighted
sums, the results in the remaining sections are of example-type leaving
space for further investigations.}

\section{The generalized cancellation property}  
\label{sec:cancellation} 
 
\setcounter{equation}{0} 
 
Let $\rho$ and $\nu$ be $\sigma$-finite measures on $\bbr^d$.  
We define the {\it multiplicative convolution} of $\nu$ and $\rho$  
as a (not necessarily $\sigma$-finite) measure on $\bbr^d$ given by  
\beao 
(\nu \circledast \rho)\, (B) = \int_{\bbr^d} \nu\bigl( T_\bx^{-1}(B)\bigr)\, 
\rho(d\bx), \ \ \text{any  Borel set $B\subset \bbr^d$\,,} 
\eeao 
where $T_\bx={\rm diag}(\bx)$ for $\bx\in\bbr^d$. 
 
We start with the following result that will motivate the cancellation 
property discussion in the sequel. 
\begin{theorem} \label{t:subseq.char} 
Assume $\alpha>0$ and let $\rho$, $\nu$ be $\sigma$-finite 
measures, such that $\rho$ is not concentrated on any proper 
coordinate subspace of $\bbr^d$, that is, 
\begin{equation} \label{e:rho.nond} 
\inf_{j=1,\ldots,d}\rho\bigl( \{ \bx:\, x_j\not=0\}\bigr)>0 \,, 
\end{equation}  
$\nu \circledast \rho\in {\rm RV}(\mu,\alpha)$ and for some 
$0<\delta'<\alpha$, 
\begin{equation} \label{e:small.rho} 
\int_{\bbr^d} \bigl( \|\by\|^{\alpha-\delta'}\vee 
\|\by\|^{\alpha+\delta'}\bigr)\, \rho(d\by)<\infty\,, 
\end{equation} 
and for each $j=1,\ldots, d$, 
\begin{equation} \label{e:small.tail} 
\lim_{b\downarrow 0}\limsup_{s\to\infty} \frac{\int_{0<|y_j|\leq b} 
  \rho\bigl( \{ \bx:\, |x_j|>s/|y_j|\}\bigr)\, \nu(d\by)}{(\nu 
  \circledast \rho) \bigl( \{ \bx:\,  \|\bx\|>s\}\bigr)}=0\,. 
\end{equation} 
Then the family of measures on $\bbr_0^d$ given by 
\begin{equation} \label{e:tight.fam} 
\mu_s(\cdot) = \frac{\nu(s\cdot)}{(\nu 
  \circledast \rho) \bigl( \{ \bx:\,  \|\bx\|>s\}\bigr)}, \ \ s\geq 
1\,, 
\end{equation} 
is relatively compact in the vague topology on $\ov\bbr_0^d$. Further, any 
limiting (as $s\to\infty$) point $\mu_\ast$ does not charge the set of infinite 
points and satisfies the equation  
\begin{equation} \label{e:right.rel} 
\mu_\ast  \circledast \rho = \mu\,. 
\end{equation} 
\end{theorem}  
\begin{proof} 
By \eqref{e:rho.nond},  we can choose  
$\theta>0$ such that $\rho\bigl( \{ \bx:\, |x_j|\geq 
\theta\}\bigr)\geq \delta>0$ for every $j=1,\ldots, d$, {\red and a
sufficiently small 
$\delta$.} For every  
$j$ and $s>0$,   
$$ 
\rho\bigl( \{ \bx:\, |x_j|\geq \theta\}\bigr) \nu\bigl( \{ \bx:\, |x_j|>s/\theta 
\}\bigr) \leq (\nu \circledast \rho) \bigl( \{ \bx:\, 
\|\bx\|>s\}\bigr)\,. 
$$ 
Therefore, 
\beao 
\lefteqn{\nu\bigl( \{ \bx:\,  \|\bx\|>s\}\bigr) 
\leq \sum_{j=1}^d \nu\bigl( \{ \bx:\, 
|x_j|>s/d\}\bigr)}\\ 
&\le & 
(\nu \circledast \rho) \bigl( \{ \bx:\, \|\bx\|>\theta 
s/d\}\bigr) \sum_{j=1}^d \frac{1}{\rho\bigl( \{ \bx:\, |x_j|\geq 
  \theta\}\bigr) }\\ 
&\leq& \frac{d}{\delta} (\nu \circledast \rho) \bigl( \{ \bx:\, \|\bx\|>\theta s/d\}\bigr)\,. 
\eeao 
By $B_\tau$, we denote  the closed ball of 
radius $\tau>0$ centered at the origin.  
 Then we have as $s\to\infty$, 
\beao 
\mu_s\bigl( B_\tau^c\bigr) \leq \frac{d}{\delta}  \frac{(\nu \circledast \rho) \bigl( \{ \bx:\, \|\bx\|>\theta 
\tau s/d\}\bigr)}{(\nu\circledast \rho) \bigl( \{ \bx:\,  \|\bx\|>s\}\bigr)} 
\to \frac{d}{\delta} \bigl(\frac{\theta\tau}{d}\bigr)^{-\alpha} \mu(B_1)<\infty\,. 
\eeao Hence  
$(\mu_s)$ is relatively {\red compact} (see Proposition 3.16 in 
\cite{resnick:1987}).
{\red Let  
$s_k\to\infty$ be a \seq\ such that $\mu_{s_k}\stv \mu_\ast$ in $\ov   
\bbr^d_0   $ for some limiting point $\mu_\ast$.} Then  $\mu_\ast$ does
not charge the set of infinite points. 
For $\bfa\in \bbr^d$ denote 
$$ 
D_{\bfa} = \Bigl\{ \bfy \in \bbr_0^d: \mu_\ast\Big(\bigl\{ \bz:\, z_j = a_j/y_j \ \ \text{for some  $j=1,\ldots, d$}\bigr\}\Big)=0\Bigr\}. 
$$ 
The argument after (2.22) in 
\cite{jacobsen:mikosch:rosinski:samorodnitsky:2008} shows that there 
are at most countable sets $A_1,\ldots, A_d$  
of real numbers such that  
\begin{equation} \label{e:not.in.A} 
\rho\bigl( D_{\bfa}\bigr) = 0, \quad \bfa \in \prod_{j=1}^d A_j^c.  
\end{equation} 
Consider $\bfa$ \st  
\beam\label{eq:kk} 
a_1>0, \quad  a_j\geq 0,\quad j=2,\ldots, d, \quad a_j\notin 
  A_j, \quad j=1,\ldots, d\,. 
\eeam   
The set $  C_d(\bfa)=\prod_{j=1}^d [a_j, 
\infty)$ is bounded away from the origin and a continuity set for 
the tail measure $\mu$ of $\nu \circledast \rho$. Therefore 
\beam \label{e:decomp} 
\mu( C_d(\bfa))\nonumber &=&  
\lim_{k\to\infty} \frac{(\nu \circledast \rho) \Bigl( s_k 
  C_d(\bfa)\Bigr)}{(\nu\circledast \rho) \bigl( \{ 
  \bx:\,  \|\bx\|>s_k\}\bigr)}\nonumber\\  
&=& \lim_{k\to\infty} \sum_{I_+\subseteq \{1,\ldots, d\}}\int_{D(I_+)} 
f_{k,I_+}(\bfz) 
\, \rho(d\bz)\,, 
\eeam 
where for $I_+\subseteq \{1,\ldots, d\}$, 
$$ 
D(I_+) = \bigl\{ \bz:\, z_j\geq 0 \ \ \text{for $j\in I_+$ and} \ \ 
z_j<0 \ \ \text{for $j\notin I_+$}\bigr\}\,, 
$$ 
interpreting $[0/0,\infty)=\bbr$ and writing  for $k\ge 1$ and
$v$ such that $v_j\geq 0$ for $j\in I_+$, 
\beao 
f_{k,I_+}(v)&=&  \mu_{s_k} \Big(  
\prod_{j\in I_+}\bigl[ a_j/v_j,\infty\bigr)\times \prod_{j\notin 
  I_+}\bigl( -\infty,-a_j/|v_j|\bigr]\Big)\,. 
\eeao 
Choosing $\vep>0$ so small that  
$ 
c:=\rho\bigl(\{ \bz:\, |z_1|\geq \vep\}\bigr)>0\,, 
$ 
and proceeding similarly to the beginning of the proof, we get 
for $I_+\subseteq \{1,\ldots, d\}$ and ${\bfz}={\bf1}=(1,\ldots,1)$, 
\beao 
f_{k,I_+}(\bf1)   
&\leq& \mu_{s_k} \bigl( \{ \by:\, |y_1|>a_1\}\bigr)\nonumber\\ 
&\leq& c^{-1} \frac{(\nu \circledast \rho) \bigl( \{ \bx:\,  |x_1|\geq 
a_1\vep s_k\}\bigr)}{(\nu\circledast \rho) \bigl( \{ \bx:\, 
\|\bx\|>s_k\}\bigr)}\,.
\eeao 
  Therefore, on $D(I_+)\cap \{ 
\bz:\, |z_1|\leq Ms_k\}$, $M>0$,  
\beao 
f_{k,I_+}(\bfz) 
&\leq& c^{-1} \frac{(\nu \circledast \rho) \bigl( \{ \bx:\,  |x_1|\geq 
a_1\vep s_k/z_1\}\bigr)}{(\nu\circledast \rho) \bigl( \{ \bx:\, 
\|\bx\|>s_k\}\bigr)}\\ 
&\leq& C(a_1,\vep,M) \bigl( |z_1|^{\alpha-\delta'}\vee 
|z_1|^{\alpha+\delta'}\bigr)\,.  
\eeao
Here $C(a_1,\vep,M)$ is a finite positive constant, and in the last 
step we used the Potter bounds; see Proposition 0.8 in 
\cite{resnick:1987}.  
Recalling that \eqref{e:not.in.A} holds for our choice of $\bfa$, 
using \eqref{e:small.rho} and the dominated convergence 
theorem, we conclude that for every $M>0$, as $k\to\infty$, 
\beao 
\lefteqn{\int_{D(I_+)\cap \{ \bz:\, |z_1|\leq Ms_k\}} f_{k,I_+}(\bfz) 
\,\rho(d\bz)}\\   
&\to& 
\int_{D(I_+)} \mu_{\ast} \left(  
\prod_{j\in I_+}\bigl[ a_j/z_j,\infty\bigr)\times \prod_{j\notin 
  I_+}\bigl( -\infty,-a_j/|z_j|\bigr]\right)\, \rho(d\bz)  \,. 
\eeao 
Furthermore,  
\beao 
\lefteqn{\int_{D(I_+)\cap \{ \bz:\, |z_1|> Ms_k\}} f_{k,I_+}(\bfz) 
\, \rho(d\bz)}\\   
&\leq& \frac{\rho\bigl( \{ \bz:\, |z_1|>Ms_k\}\bigr)\nu\bigl( \{ \by:\, |y_1|>a_1/M\}\bigr)} 
{(\nu\circledast \rho) \bigl( \{ \bx:\, \|\bx\|>s_k\}\bigr)}\\ 
&& 
+ \frac{\int_{0<|y_1|\leq a_1/M} \rho\bigl( \{ \bz:\, |z_1|>s_k a_1/|y_1|\}\bigr)\, \nu(d\by)}{(\nu\circledast \rho) \bigl( \{ \bx:\,  \|\bx\|>s_k\}\bigr)} \, := A_k+B_k\,. 
\eeao 
Since  
$$ 
\rho\bigl( \{ \bz:\, |z_1|>Ms_k\}\bigr) \leq (Ms_k)^{-(\alpha+\delta)}\int_{\bbr^d} 
|z_1|^{\alpha+\delta}\, \rho(d\bz)\,, 
$$ 
it follows from \eqref{e:small.rho} that $A_k\to 0$ as $k\to\infty$, 
once again for each $M>0$, and by 
\eqref{e:small.tail}, 
$ 
\lim_{M\to\infty}\limsup_{k\to\infty} B_k=0 
$. 
Thus we proved that for any $\bfa$ satisfying \eqref{eq:kk} 
and $I_+\subseteq \{1,\ldots, d\}$, as $k\to\infty$, 
\beao 
\lefteqn{ 
\int_{D(I_+)} f_{k,I_+}(\bfz)\, \rho(d\bz)}\\   
&\to& 
\int_{D(I_+)} \mu_{\ast} \left(  
\prod_{j\in I_+}\bigl[ a_j/z_j,\infty\bigr)\times \prod_{j\notin 
  I_+}\bigl( -\infty,-a_j/|z_j|\bigr]\right)\, \rho(d\bz) \,.  
\eeao 
Then, in view of \eqref{e:decomp}, 
\beam \label{e:quadrant.1} 
&&\mu( C_d(\bfa))\nonumber\\  
&=& \sum_{I_+\subseteq 
  \{1,\ldots, d\}}\int_{D(I_+)} \mu_{\ast} \left(  
\prod_{j\in I_+}\bigl[ a_j/z_j,\infty\bigr)\times \prod_{j\notin 
  I_+}\bigl( -\infty,-a_j/|z_j|\bigr]\right)\, \rho(d\bz)\nonumber  \\ 
&=& (\mu_\ast  \circledast \rho)(C_d(\bfa)) \,. 
\nonumber\\\eeam 
Using the continuity of measures from above, we can now extend 
\eqref{e:quadrant.1} to any $\bfa$ satisfying  
 $a_1>0, \, a_j\geq 0, j=2,\ldots, d$. This 
means that the measures $\mu$ and $\mu_\ast  \circledast \rho$ 
coincide on the set $\{ \bx:\, x_1>0,\, x_j\geq 0, j=2,\ldots, d\}$.  
\par 
Of course, this argument can be repeated while distinguishing any 
coordinate $k=1,\ldots, d$, so that we see that the measures  
$\mu$ and $\mu_\ast \circledast \rho$ 
coincide on each of the $d$ sets $\{ \bx:\, x_k>0,\, x_j\geq 0, 
j=1,\ldots, d\}$, $k=1,\ldots, d$. Since the union of these sets is 
the first quadrant $[0,\infty)^d\backslash \{\bf0\}$ we conclude that 
these two measures coincide on this set.  
An identical argument can be used for all other quadrants of 
$\ov\bbr_0^d$. Thus \eqref{e:right.rel} holds and the 
proof of the theorem is complete.  
\end{proof} 
There is only an apparently small step remaining between 
the conclusion of Theorem \ref{t:subseq.char} and the statement that 
$\nu$ is regularly varying with index $\alpha$. This 
step consists of showing that (with $\rho$ and $\mu$ fixed)  
equation \eqref{e:right.rel} has a unique solution 
$\mu_\ast$. Indeed, if this could be established, then all 
subsequential limits as $s\to\infty$ of the family $(\mu_s)$ in 
\eqref{e:tight.fam} would be equal. In turn, 
$(\mu_s)$ would converge vaguely and $\nu$ would be \regvary . 
 
Unfortunately, this step is not so small and it turns out that, 
in some cases,   
\eqref{e:right.rel}  has multiple 
solutions; {\red see the  following discussion and, in particular,
Remark~\ref{ex:extra}.}
Therefore, 
our next step aims at establishing conditions under 
which the solution to \eqref{e:right.rel} is, indeed, 
unique. We start by reducing the problem to a slightly different 
form. Uniqueness of the solution to \eqref{e:right.rel} 
would follow if the measure $\rho$ had the following property: within 
a relevant class of $\sigma$-finite measures $\nu_1,\nu_2$,  
\begin{equation} \label{e:gen.cancel} 
\text{if} \ \ \nu_1   \circledast \rho = \nu_2  \circledast \rho \ \ 
\text{then} \ \ \nu_1=\nu_2\,. 
\end{equation} 
This property can be viewed as the {\em cancellation property  
of the measure $\rho$ with respect to the operation $\circledast$}.  
 
A similar situation was considered in 
\cite{jacobsen:mikosch:rosinski:samorodnitsky:2008}, in which the 
case $d=1$ was treated. There it was assumed that all measures 
are supported on the positive half-line $(0,\infty)$. In particular, 
all regularly varying measures supported on $(0,\infty)$ 
have tail measures proportional to one another. It is 
natural in this situation to study the cancellation property if one 
of the measures $\nu_1, \nu_2$ is such a canonical 
measure. Correspondingly, one defines a  measure $\nu^\alpha$ on 
$(0,\infty)$, $\alpha\in\bbr$, with a power density given by  
\begin{equation} \label{e:nu.alpha} 
\nu^{\alpha}(dx)= 
\alpha \,x^{-(\alpha+1)}\,dx \,. 
\end{equation} 
Actually, \cite{jacobsen:mikosch:rosinski:samorodnitsky:2008} allow 
at this point for any real value of $\alpha$. In the present paper, we will look 
only at positive $\alpha$, even though the statement of Theorem 
\ref{prop:2} below can be extended to the more general case.  
\par 
The paper \cite{jacobsen:mikosch:rosinski:samorodnitsky:2008} 
addresses the question which measures $\rho$ have the following 
cancellation property:  
$$ 
\nu \circledast \rho = \nu^\alpha \circledast \rho \quad 
        \text{ implies  $\nu=\nu^\alpha$,} 
$$ 
and it was  shown that a measure  
$\rho$ satisfying 
$$ 
\int_0^\infty y^{\alpha-\delta}\vee y^{\alpha+\delta}\, 
\rho(dy)<\infty\,,\quad \mbox{for some $\delta>0$,}  
$$  
has this cancellation property if and only if  
$$ 
\int_{0}^{\infty} y^{\alpha +i\theta} \, \rho(dy) \ne 0  \quad 
\text{for all } \, \theta \in \bbr\,. 
$$ 
 
In order to understand  the more general cancellation 
property \eqref{e:gen.cancel}, we start by replacing the single 
equation by a system of linear equations that include only measures 
concentrated on the positive quadrant of $\bbr^d$.  
\par
For $d\geq 1$, consider the set $Q_d=\{-1,1\}^d$ equipped with the 
coordinate-wise (binary) multiplication. Let $\alpha_1,\ldots, 
\alpha_d$ be positive  numbers, $\bigl( \rho_v, \, v\in Q_d\bigr)$ be 
$\sigma$-finite measures on $(0,\infty)^d$, and  
$\bigl( \nu_v^{(i)}, \, v\in Q_d\bigr), \, i=1,2$, be two collections 
of $\sigma$-finite measures on $[0,\infty)^d$. We assume that for a certain non-empty subset $K$  of $\{ 1,\ldots, d\}$ 
\begin{equation} \label{e:mom.cond.rho} 
\int_{(0,\infty)^d} x_j^{\alpha_j}\, \rho_v(d\bx)<\infty \ \ \text{for 
  each $v\in Q_d$ and $j\in K$,} 
\end{equation} 
 and for $i=1,2$, 
\begin{equation} \label{e:tail.cond.nu} 
\sup_{s>0}\, s^{\alpha_j}\nu_v^{(i)}\bigl( \{\bx:\, x_j>s\}\bigr)<\infty 
\ \ \text{for 
  each $v\in Q_d$ and $j\in K$.} 
\end{equation}  
We now assume that these measures satisfy the following system of 
$2^d$ linear equations. 
\begin{equation} \label{e:main.system} 
\sum_{w\in Q_d} \nu_w^{(1)}\circledast \rho_{vw} 
=\sum_{w\in Q_d} \nu_w^{(2)}\circledast \rho_{vw} 
\ \ \text{for each $v\in Q_d$.} 
\end{equation} 
 
The following result characterizes those measures $\bigl( \rho_v, \, v\in 
Q_d\bigr)$ which can be ``cancelled'' in this system of equations.  
\bth\label{prop:2} 
Let  $\bigl( \rho_{v}, \, {v}\in Q_d\bigr)$  be 
$\sigma$-finite measures on $(0,\infty)^d$  and  
$\bigl( \nu_ {v}^{(i)}, \, {v}\in Q_d\bigr), \, i=1,2$, be  
$\sigma$-finite measures on $[0,\infty)^d$. Assume that for some 
non-empty set 
$K\subseteq \{ 1,\ldots, d\}$, 
\begin{equation} \label{e:set.K} 
\nu_{v}^{(i)}\bigl( \{ \bx:\, x_k=0 \ \ \text{for each $k\in 
  K$}\}\bigr)=0\,,\quad i=1,2,\quad {v}\in Q_d\,, 
\end{equation} 
and that 
\eqref{e:mom.cond.rho} and \eqref{e:tail.cond.nu} hold for this set 
$K$.  
Suppose that for each $j\in K$, $m_1,\ldots, m_d\in \{0,1\}$ and 
$\theta_1,\ldots, \theta_d\in \bbr$, 
\begin{equation} \label{e:cancel.condition} 
\sum_{{v}\in Q_d}\prod_{k=1}^d v_k^{m_k}\int_{(0,\infty)^d} 
x_j^{\alpha_j} \prod_{k=1}^d x_k^{i\theta_k} \rho_v(d\bx)\not= 0 
\end{equation} 
with the usual notation $   v=(v_1,\ldots, v_d)\in Q_d$ and 
$\bx=(x_1,\ldots, x_d)\in [0,\infty)^d$.  
If these measures satisfy 
the system of equations \eqref{e:main.system}, then  
\begin{equation} \label{e:equal.nu} 
\nu_{v}^{(1)}=\nu_{v}^{(2)} \ \ \text{for each $v\in Q_d$.} 
\end{equation} 
\ethe 
\medskip 
 
\begin{remark} \label{rk:K} 
{\rm 
In applications to regular variation the measures $\bigl( 
\nu_v^{(i)}, \, v\in Q_d\bigr), \, i=1,2$, will appear as (restrictions 
to the different quadrants of) certain vague 
limits $\nu$ in $\ov\bbr_0^d$, hence will automatically put no mass at the 
origin. Hence the set $K=\{ 1,\ldots, d\}$ will always satisfy 
\eqref{e:set.K}. This is the maximal possible choice of $K$ which 
requires the largest possible set of conditions in 
\eqref{e:cancel.condition}. {\red The smaller the set $K$ can be
  chosen, 
the fewer  conditions one needs to check.} 
If, for example, $\nu$ is absolutely continuous, then $K=\{1\}$ 
and \eqref{e:set.K} gives $2^d$ conditions.  
} 
\end{remark} 
 
Before proving Theorem \ref{prop:2}, we consider some special cases.  We start by considering the 
scalar case, $d=1$. In this case, the system of equations 
\eqref{e:main.system} becomes 
\begin{equation} \label{e:d.1} 
\nu_1^{(1)}\circledast \rho_{1} + \nu_{-1}^{(1)}\circledast \rho_{-1} 
=\nu_1^{(2)}\circledast \rho_{1} + \nu_{-1}^{(2)}\circledast 
\rho_{-1}\,, 
\end{equation} 
$$ 
\nu_1^{(1)}\circledast \rho_{-1} + \nu_{-1}^{(1)}\circledast \rho_{1} 
=\nu_1^{(2)}\circledast \rho_{-1} + \nu_{-1}^{(2)}\circledast 
\rho_{1}\,. 
$$ 
The only choice is $K=\{ 1\}$ and the conditions 
\eqref{e:cancel.condition} for the cancellation property become 
\beam 
\left\{  \barr{l} 
\int_0^\infty x^{\alpha_1+i\theta}\, \rho_1(dx) 
+\int_0^\infty x^{\alpha_1+i\theta}\, \rho_{-1}(dx)\not= 0\,,\\ 
\int_0^\infty x^{\alpha_1+i\theta}\, \rho_1(dx) 
-\int_0^\infty x^{\alpha_1+i\theta}\, \rho_{-1}(dx)\not= 
0\,,\earr 
\right.\,,\quad  
\theta\in \bbr\,. \label{e:cancel.cond.d1} 
\eeam 
In dimension one the measure $\nu^\alpha$, $\alpha>0$, given in 
\eqref{e:nu.alpha}, is particularly important when studying regular 
variation. Suppose that $\nu_i^{(2)}=c_i \nu^\alpha$, $i=\pm 1$, where 
$c_1\,, c_{-1}$ are nonnegative constants. If we 
choose $\alpha_1= \alpha$, then the assumption 
\eqref{e:tail.cond.nu} automatically holds for the measures 
$\nu_1^{(2)}$ and $\nu_{-1}^{(2)}$. Assuming that the measures 
$\rho_1$, $\rho_{-1}$ satisfy \eqref{e:mom.cond.rho} and  
$ 
\|\rho_{i}\|_\alpha^\alpha= \int_0^\infty x^\alpha\, \rho_i(dx), \ \ 
i=\pm 1\,, 
$ 
the system 
\eqref{e:d.1} takes the form 
\begin{equation} \label{e:d.1.alpha} 
\nu_1^{(1)}\circledast \rho_{i} + \nu_{-1}^{(1)}\circledast \rho_{-i} 
= \bigl( c_1\|\rho_{i}\|_\alpha^\alpha +c_{-1} 
\|\rho_{-i}\|_\alpha^\alpha\bigr)\;\nu^\alpha, \ \ i=\pm 1\,. 
\end{equation} 
Notice that the two equations \eqref{e:d.1.alpha} already imply that  
\eqref{e:tail.cond.nu}  holds for the measures 
$\nu_1^{(1)}$ and $\nu_{-1}^{(1)}$ as well. We therefore obtain the 
following corollary of Theorem \ref{prop:2}.  
\begin{corollary} \label{cor:two.sided} 
Let $\alpha_1=\alpha>0$, and $\rho_1$, $\rho_{-1}$ be $\sigma$-finite 
measures on $(0,\infty)$ satisfying \eqref{e:mom.cond.rho}. If the 
$\sigma$-finite measures on $[0,\infty)$, $\nu_1^{(1)}$ and 
$\nu_{-1}^{(1)}$, satisfy the system of equations \eqref{e:d.1.alpha}, 
and if the cancellation conditions \eqref{e:cancel.cond.d1} are 
satisfied, then $\nu_{i}^{(1)}=c_i\nu^\alpha$, $i=\pm 1$.  
\end{corollary} 
\bre\label{ex:extra}{\red \em
Assume that all conditions of Corollary~\ref{cor:two.sided} 
but \eqref{e:cancel.cond.d1}
are
satisfied. For example, if the first condition in
\eqref{e:cancel.cond.d1}
is not satisfied for some $\theta=\theta_0\in\bbr$, then the \ms s 
\beao
\nu_i^{(1)}(dx)= [c_i+ a\,\cos(\theta_0\log x) + b \sin(\theta_0\log
x)]\,\nu^\alpha(dx)\,,\quad i=\pm 1\,,
\eeao
for $a,b$ \st\ $0\le a^2+b^2\le 1$ solve the system of equations
 \eqref{e:d.1.alpha}. Similarly, if the second condition in \eqref{e:cancel.cond.d1}
fails for some $\theta=\theta_0\in\bbr$, then the \ms s
\beao
\nu_i^{(1)}(dx)= \big[c_i+ (-1)^i\big(a\,\cos(\theta_0\log x) + b\, \sin(\theta_0\log
x)\big)\big]\,\nu^\alpha(dx)\,,\quad i=\pm 1\,,
\eeao
with the same choice of $a,b$ as above satisfy \eqref{e:d.1.alpha}.
}
\ere
Another useful special case of Theorem \ref{prop:2} corresponds to the 
situation, where only one of the measures $\bigl( \rho_v, \, v\in 
Q_d\bigr)$ is non-null;  
as we will see in the sequel this case  
naturally arises in inverse problems for regular variation.  
 We assume without loss of generality that the 
non-null measure corresponds to the unity in $Q_d$, $v=(1,\ldots, 
1)$. For simplicity denoting this measure by $\rho$, 
we see that the system of equations \eqref{e:main.system} decouples, 
and becomes 
$$ 
\nu_v^{(1)}\circledast \rho 
=  \nu_v^{(2)}\circledast \rho \ \ \text{for each $v\in Q_d$.} 
$$ 
However, the decoupled system of equations does not provide us 
with any additional insight over a single equation, so the right thing to 
do is to drop the subscript and consider the equation  
\begin{equation} \label{e:pos.case.eqn} 
\nu^{(1)}\circledast \rho 
=  \nu^{(2)}\circledast \rho 
\end{equation} 
for two $\sigma$-finite measures $\nu^{(1)}$ and $\nu^{(2)}$.  
If we interpret \eqref{e:mom.cond.rho}, \eqref{e:tail.cond.nu} and 
\eqref{e:set.K} by disregarding the subscripts, we obtain another  
corollary of Theorem \ref{prop:2}.  
\begin{corollary} \label{cor:positive} 
Let $\alpha_1,\ldots, \alpha_d$ be positive numbers, $\rho$ a 
$\sigma$-finite measure on $(0,\infty)^d$ and  $\nu^{(1)}$, 
$\nu^{(2)}$ $\sigma$-finite measures on $[0,\infty)^d$. Suppose 
that the non-empty set 
$K\subseteq \{ 1,\ldots, d\}$ satisfies \eqref{e:set.K} and  
\eqref{e:mom.cond.rho} and \eqref{e:tail.cond.nu} hold.  
 
If the equation \eqref{e:pos.case.eqn} is fulfilled, and  
\begin{equation} \label{e:cancel.condition.pos} 
\int_{(0,\infty)^d} 
x_j^{\alpha_j} \prod_{k=1}^d x_k^{i\theta_k} \rho(d\bx)\not= 0 
\end{equation}  
for each $j\in K$,  and 
$\theta_1,\ldots, \theta_d\in \bbr$, then $\nu^{(1)}=\nu^{(2)}$.  
\end{corollary} 
 
In the case $d=1$, the conclusion of Corollary \ref{cor:positive} is 
the same as the direct part of Theorem 2.1 in 
\cite{jacobsen:mikosch:rosinski:samorodnitsky:2008}.

\begin{proof}[Proof of Theorem \ref{prop:2}] 
The general idea of the proof is similar to the proof of Theorem 2.1 
in \cite{jacobsen:mikosch:rosinski:samorodnitsky:2008}. Fix $j\in K$ 
and define 
$$ 
h_j^{(v,i)}(\by) = y_j^{\alpha_j}\nu_v^{(i)}\bigl( \{ \bz:\, 0\leq 
z_k\leq y_k, \, k\not= j, \, z_j>y_j\}\bigr), \ \ v\in Q_d, \, i=1,2 
$$ 
for $\by=(y_1,\ldots, y_d)$ with all $y_k>0$. It 
follows from \eqref{e:tail.cond.nu} that all these functions  
are bounded on their domain. The equations \eqref{e:main.system} then 
tell us that  
\beao 
\lefteqn{\sum_{w\in Q_d} \int_{[0,\infty)^d} h_j^{(w,1)}\bigl( 
x_1/z_1,\ldots, x_d/z_d\bigr)\, \rho_{vw}(d\bz)}\\ 
&=&\sum_{w\in Q_d} \int_{[0,\infty)^d} h_j^{(w,2)}\bigl( 
x_1/z_1,\ldots, x_d/z_d\bigr)\, \rho_{vw}(d\bz) 
\eeao 
for each $v\in Q_d, \, x_k>0, \, k=1,\ldots, d$. Next, we define 
functions  
$$ 
g_j^{(v,i)}(\by) = h_j^{(v,i)}\bigl( e^{y_1},\ldots, e^{y_d}\bigr), \ \ v\in Q_d, \, i=1,2, 
$$ 
for  $\by\in \bbr^d$, and finite measures on 
$\bbr^d$ by  
$$ 
\mu_j^{(v)}(d\bx) = \big(e^{\alpha_jx_j}\rho_v\big)\circ T_{\log}^{-1}(d\bx)\,, 
$$ 
where $T_{\log}(\by)=(\log y_1, \ldots, \log y_d)$, $\by\in 
(0,\infty)^d$. We can now write 
$$ 
\sum_{w\in Q_d} \int_{\bbr^d} g_j^{(w,1)}(\bz-\by)\, \mu_j^{(vw)}(d\by)  
= \sum_{w\in Q_d} \int_{\bbr^d} g_j^{(w,2)}(\bz-\by)\, \mu_j^{(vw)}(d\by)  
$$ 
for each $v\in Q_d, \, \bz\in\bbr^d$. Therefore, the bounded functions 
$$ 
g_j^{(v)}(\by) = g_j^{(v,1)}(\by) - g_j^{(v,2)}(\by), \ \by\in\bbr^d, 
\ \ v\in Q_d, 
$$ 
satisfy  
\begin{equation} \label{conv.eq.v} 
\sum_{w\in Q_d} \int_{\bbr^d} g_j^{(w)}(\bz-\by)\, \mu_j^{(vw)}(d\by) 
=0 
\end{equation} 
for each $v\in Q_d, \, \bz\in\bbr^d$. 
 
For fixed  $m_k\in \{0,1\}, \, k=1,\ldots, d$, and $j\in K$, we define a signed bounded measure 
on $\bbr^d$ by  
$$ 
\mu_j = \sum_{v\in Q_d} \prod_{k=1}^d v_k^{m_k}\mu_j^{(v)} 
$$ 
and a bounded function on $\bbr^d$ by  
$$ 
g_j= \sum_{v\in Q_d} \prod_{k=1}^d v_k^{m_k}g_j^{(v)}\,. 
$$ 
Then the system of equations \eqref{conv.eq.v} implies 
\begin{equation} \label{conv.eq} 
\int_{\bbr^d} g_j(\bz-\by)\, \mu_j(d\by) 
=0\,,\quad \bz\in\bbr^d\,, 
\end{equation} 
and we want to show that $g_j=0$ everywhere. 
 
Notice now that the right-hand side of \eqref{e:cancel.condition} is exactly the Fourier transform of $\mu_j$ at the point $\bfs = (\theta_1,\ldots,\theta_d)$.  
Let $\varphi$ be the standard normal density in 
$\bbr^d$. Then, in the standard notation for the additive 
convolution, we have $\varphi \ast \mu_j \in L^1(\bbr^d)$, and the 
equation \eqref{conv.eq} tells us that $g_j\ast \bigl( \varphi \ast 
\mu_j\bigr)\equiv 0$. Let the symbol \ $\wh\ $\ \  denote the distributional Fourier transform of 
a function or a signed measure.  
By Theorem 9.3 in \cite{rudin:1973} we have 
that, in the distributional sense,  
$$ 
{\rm supp}(\wh g_j)\subseteq \{\bfs\in \bbr^d: \wh 
\varphi(\bfs) \wh \mu_j(\bfs)=0\}= \{\bfs\in \bbr^d: \wh \mu_j(\bfs)=0\} = \emptyset\,, 
$$ 
where the last equation is just the condition \eqref{e:cancel.condition}.  
Therefore, we conclude 
that the support of the Fourier transform $\wh g_j$ is empty, hence 
$g_j=0$ almost everywhere. Since the function $g_j$ is coordinate-wise 
right-continuous, we see that $g_j=0$ everywhere. 
 
The $2^d\times 2^d$ matrix $A$ with the entries  
$$ 
a_{m_1,\ldots,m_d, v_1,\ldots, v_d} = \prod_{k=1}^d v_k^{m_k}, \ \ 
m_j\in\{0,1\}, \, v_j\in \{-1,1\}, \ \ j=1,\ldots,d\,, 
$$ 
is non-degenerate (in fact, $|{\rm det} \, A| = 
2^{d2^{d-1}}$). Therefore, it follows from the definition of the 
function $g_j$ that for each $v\in Q_d$, $g_j^{(v)}\equiv 0$, hence  
$g_j^{(v,1)}\equiv g_j^{(v,2)}$. We conclude that  
\beao 
\lefteqn{\nu_v^{(1)}\bigl( \{ \bz:\, 0\leq 
z_k\leq y_k, \, k\not= j, \, z_j>y_j\}\bigr)}\\ 
&=& \nu_v^{(2)}\bigl( \{ \bz:\, 0\leq 
z_k\leq y_k, \, k\not= j, \, z_j>y_j\}\bigr)  
\eeao 
for each $v\in Q_d$, $\by\in (0,\infty)^d$ and $j\in K$. This means 
that, for each $v\in Q_d$,  the 
measures $\nu_v^{(1)}$ and $\nu_v^{(2)}$ coincide on the set $\{ 
y_j>0\}$ for each $j\in K$. By the definition of the set $K$ we obtain 
\eqref{e:equal.nu} and, hence, complete the proof.  
\end{proof} 
 
The conditions for the cancellation property in \eqref{e:cancel.condition} 
and its special cases above, are somewhat implicit.  
 On the other hand, 
in the case of one dimension and a single equation, the presence 
of a sufficiently large atom in the measure $\rho$  
already guarantees the cancellation property; see Corollary 2.2 in 
\cite{jacobsen:mikosch:rosinski:samorodnitsky:2008}. A similar 
phenomenon, described in the following statement, occurs in general.  
 
\begin{corollary} \label{c:atoms.cancel} 
Let  $\bigl( \rho_v, \, v\in Q_d\bigr)$  be 
$\sigma$-finite measures on $(0,\infty)^d$,  and let  
$\bigl( \nu_v^{(i)}, \, v\in Q_d\bigr), \, i=1,2$, be  
$\sigma$-finite measures on $[0,\infty)^d$. Suppose that $K$ is a 
nonempty set satisfying \eqref{e:set.K}. Assume, further, that  
\eqref{e:mom.cond.rho} and \eqref{e:tail.cond.nu} hold for this set 
$K$.   
 
Suppose that these measures satisfy 
the system of equations \eqref{e:main.system}. If for every $j\in K$ 
there is $v^{(j)}\in Q_d$ and an atom 
$\bx^{(j)}=\bigl(x^{(j)}_1,\ldots,  x^{(j)}_d\bigr)$ of $\rho_{v^{(j)}}$ 
with mass $w^{(j)}$ so large that 
$$ 
w^{(j)} \bigl( x^{(j)}_j\bigr)^{\alpha_j} > \int_{\bx\not= \bx^{(j)}} 
x_j^{\alpha_j}\, \rho_{v^{(j)}}(d\bx) + \sum_{v\not=v^{(j)}}  
\int_{(0,\infty)^d} x_j^{\alpha_j}\, \rho_{v}(d\bx) \,, 
$$ 
then the conclusion \eqref{e:equal.nu} holds.  
\end{corollary} 
\begin{proof} 
An application of the triangle inequality shows that the 
assumptions of the corollary, in fact, imply 
\eqref{e:cancel.condition}. Indeed, let $j\in K$. We have, for any 
$m_1,\ldots, m_d\in \{0,1\}$ and $\theta_1,\ldots, \theta_d\in \bbr$,   
\beao 
\lefteqn{\left| \sum_{v\in Q_d}\prod_{k=1}^d v_k^{m_k}\int_{(0,\infty)^d} 
x_j^{\alpha_j} \prod_{k=1}^d x_k^{i\theta_k} \rho_v(d\bx) \right|}\\ 
&\geq& w^{(j)} \bigl( x^{(j)}_j\bigr)^{\alpha_j} - \int_{\bx\not= \bx^{(j)}} 
x_j^{\alpha_j}\, \rho_{v^{(j)}}(d\bx) - \sum_{v\not=v^{(j)}}  
\int_{(0,\infty)^d} x_j^{\alpha_j}\, \rho_{v}(d\bx) >0 
\eeao 
by the assumption, so none of the expressions in 
\eqref{e:cancel.condition} can vanish.  
\end{proof} 
 
We now put together Theorems \ref{t:subseq.char} and \ref{prop:2} and 
obtain an inverse regular variation result  
for multiplicative convolutions. It is a multivariate 
extension of Theorem 2.3 in 
\cite{jacobsen:mikosch:rosinski:samorodnitsky:2008}.  
\begin{theorem} \label{t:inverse.general} 
Let $\alpha>0$ and $\rho$, $\nu$ be $\sigma$-finite 
measures on $\bbr^d$ \st\  
$$ 
\rho\bigl( \{ \bx:\, x_i=0 \}\bigr)=0\quad\mbox{for every $i=1,\ldots, 
  d$,} 
$$ 
and $(\nu \circledast \rho)\in\RV(\alpha,\mu)$. 
Assume \eqref{e:small.rho}, \eqref{e:small.tail} and  
\begin{equation} \label{e:regvar.cond} 
\int_{\bbr^d} |x_j|^{\alpha} \prod_{k=1}^d |x_k|^{i\theta_k} 
\prod_{k=1}^d \bigl( {\rm sign}(x_k)\bigr)^{m_k}\, \rho(d\bx)\not= 0 
\end{equation} 
for each $j=1,\ldots, d$, $m_1,\ldots, m_d\in \{0,1\}$ and 
$\theta_1,\ldots, \theta_d\in \bbr$. Then the measure $\nu$ is 
regularly varying with index  $\alpha$. Moreover, the measures 
$(\mu_s)$ in \eqref{e:tight.fam} converge vaguely as $s\to\infty$, in  
$\ov\bbr_0^d$, to a measure $\mu_\ast$ satisfying 
\eqref{e:right.rel}.  
\end{theorem} 
\begin{proof} 
Because of the statement of Theorem \ref{t:subseq.char}, we only need 
to prove that any two subsequential vague limits $\nu^{(1)}$ and 
$\nu^{(2)}$  
in that theorem 
coincide.  
Note that $\nu^{(1)}$ and $\nu^{(2)}$ are two solutions to the equation 
\eqref{e:right.rel}, so in order to prove that $\nu^{(1)}=\nu^{(2)}$ 
we translate our problem to the  
cancellation property situation of Theorem \ref{prop:2}.  For $v\in 
Q_d$ denote  
$$ 
\bbq_v = \bigl\{ \bx: \, x_jv_j\geq 0 \ \ \text{for each} \ \ 
j=1,\ldots, d\bigr\}\,, 
$$ 
and define 
$$ 
\rho_v(\cdot) = \rho\bigl( \{ \bx\in \bbq_v:\, (|x_1|, \dots, 
|x_d|)\in \cdot\}\bigr)\,. 
$$  
Similarly,  
we define two collections 
of $\sigma$-finite measures on $[0,\infty)^d$, $\bigl( \nu_v^{(i)}, \, 
v\in Q_d\bigr), \, i=1,2$, by restricting the measures $\nu^{(1)}$ and 
$\nu^{(2)}$ to the appropriate quadrants. By  assumption, 
$\nu^{(1)}\circledast \rho = \nu^{(2)}\circledast \rho$. Writing up 
this equality of measures on $\bbr^d$ for each quadrant of $\bbr^d$, 
we immediately see that the measures $(\rho_v, \, v\in Q_d)$ and  
$\bigl( \nu_v^{(i)}, \, v\in Q_d\bigr), \, i=1,2$, satisfy the system 
of equations \eqref{e:main.system}. 
 
We let $\alpha_j=\alpha$ for $j=1,\ldots, d$ and $K=\{ 1,\ldots, 
d\}$. Then \eqref{e:set.K} holds since the  
measures $\nu^{(1)}$ and $\nu^{(2)}$ are vague limits in $\ov\bbr_0^d$ 
and, hence, place no mass at the origin in $\bbr^d$. The assumption  
\eqref{e:mom.cond.rho} follows from \eqref{e:small.rho}. The 
assumption \eqref{e:tail.cond.nu} follows from the fact that both  
$\nu^{(1)}$ and $\nu^{(2)}$ satisfy \eqref{e:right.rel} and the 
scaling property of the tail measure $\mu$. Finally, the condition  
\eqref{e:cancel.condition} follows from \eqref{e:regvar.cond} and 
elementary manipulation of the sums and integrals. Therefore, Theorem 
\ref{prop:2} applies, and $\nu_v^{(1)}=\nu_v^{(2)}$ for each $v\in 
Q_d$. This means that $\nu^{(1)}=\nu^{(2)}$.  
\end{proof} 
 
\begin{remark} \label{rk:check.less} 
{\rm 
If the tail measure $\mu$ of $\nu \circledast \rho$ satisfies  
\begin{equation} \label{e:nice.mu} 
\mu \bigl( \{ \bx:\, x_k=0 \ \ \text{for each $k\in 
  K$}\}\bigr)=0 
\end{equation} 
for some non-empty set $K\subseteq \{ 1,\ldots, d\}$, then every 
measure $\mu_\ast$ satisfying \eqref{e:right.rel} has the same 
property:  
$$ 
\mu_\ast \bigl( \{ \bx:\, x_k=0 \ \ \text{for each $k\in 
  K$}\}\bigr)=0\,. 
$$ 
Therefore the measures $\bigl( \nu_v^{(i)}, \, 
v\in Q_d\bigr), \, i=1,2$, defined in the proof of Theorem 
\ref{t:inverse.general} satisfy \eqref{e:set.K}, and we can apply 
Theorem \ref{prop:2} with this smaller set $K$. In other words, if 
\eqref{e:nice.mu} holds, then the condition \eqref{e:regvar.cond} in  
Theorem \ref{t:inverse.general} has 
to be checked only for $j\in K$.  
} 
\end{remark} 
 
We can extend Theorem \ref{t:inverse.general} to the situation where 
the measure $\rho$ puts a positive mass on the axes. The next 
corollary follows from the theorem by splitting the space $\bbr^d$ 
into  subspaces of different dimensions, by setting some of the 
coordinates equal to zero. We omit details. 
 
\begin{corollary} \label{c:some.zeroes} 
Let $\alpha>0$ and $\rho$, $\nu$ be $\sigma$-finite 
measures on $\bbr^d$ \st\ \eqref{e:small.rho} and \eqref{e:small.tail} 
hold and $\nu \circledast \rho\in\RV(\alpha,\mu)$. Assume that for every 
$I_0\subset \{1,\ldots, d\}$ such that  
$$ 
\rho\bigl( \bigl\{ \bfx\in \bbr^d:\, x_i=0 \ \ \text{for all $i\in 
  I_0$}\bigr\}\bigr)>0  
$$ 
we have for every $I$ such that $I_0\cup I = \{1,\ldots,d\}$, 
\begin{equation} \label{e:regvar.cond.more} 
\int_{\bbr^d} |x_j|^{\alpha} \prod_{k\in I}  |x_k|^{i\theta_k} 
\prod_{k\in I} \bigl( {\rm sign}(x_k)\bigr)^{m_k}\, \rho(d\bx)\not= 0 
\end{equation} 
for each $j\in I$, $m_k\in \{0,1\}$ and 
$\theta_k\in \bbr,$ $k\in I$. Then the conclusions of Theorem 
\ref{t:inverse.general} hold.  
\end{corollary}

\section{The inverse problem for weighted sums} \label{sec:sums} 
\setcounter{equation}{0} 
In this section we revisit the weighted sums of iid random vectors 
introduced in Example \ref{ex:linear.process}. We consider  
the special case of diagonal 
coefficient matrices. Our goal is to apply  the generalized 
cancellation theory of the previous section to investigate under what 
conditions on the coefficient matrices regular variation of the 
weighted sum implies regular variation of the 
underlying iid random vectors.  
\par 
Let $(\bfZ^{(i)})$ be an iid \seq\ of   $  \bbr^d$-valued random 
column vectors with a generic element $\bfZ$ and $(\bfd^{(i)})$ be 
deterministic vectors in  $\bbr^d$. The $i$th coefficient matrix 
$\Psi_i$ is a diagonal matrix with $\bfd^{(i)}$ on the main diagonal: 
$\Psi_i = {\rm diag}(\bfd^{(i)})$. The following theorem is the main 
result  of this section. The corresponding result for $d=1$ and 
positive weights $\psi_i$   
was proved in \cite{jacobsen:mikosch:rosinski:samorodnitsky:2008}, 
Theorem 3.3. 
\begin{theorem} \label{t:weighted.sums} 
Assume that the series $\bfX = \sum_{i=1}^\infty \Psi_i\,\bfZ^{(i)}$ 
converges a.s., $\bfX\in \RV(\alpha,\mu_{\bfX})$ 
and for some $0<\delta'<\alpha$,  
\beam\label{e:small.d} 
 \sum_{i=1}^\infty \| \bfd^{(i)}\|^{\alpha-\delta'}<\infty\,. 
\eeam 
Suppose also that all non-zero vectors $(\bfd^{(i)})$ have
non-vanishing coordinates. 
 If for all $j=1,\ldots, d$, for $m_1,\ldots, m_d\in \{0,1\}$ and 
$\theta_1,\ldots, \theta_d\in \bbr$,   
\begin{equation} \label{e:good.d} 
\sum_{l=1}^\infty \Big[ |d_j^{(l)}|^{\alpha} \prod_{k=1}^d |d_k^{(l)}|^{i\theta_k} 
\prod_{k=1}^d \bigl( {\rm sign}(d_k^{(l)} )\bigr)^{m_k} \Big] \not= 0\,, 
\end{equation} 
then $\bfZ$  
is \regvary\ with index $\alpha$ and \eqref{eq:10} holds.  
\end{theorem}
\bre{\red \em
Of course, if some of the non-zero vectors $(\bfd^{(i)})$ have 
vanishing coordinates, we can use Corollary \ref{c:some.zeroes} 
instead of Theorem \ref{t:inverse.general}, and obtain regular 
variation of the vector $\bfZ$ under a more extensive set of conditions 
than \eqref{e:good.d}.}  
\ere 
We start the proof with the following lemma. 
\begin{lemma} \label{l:sum.of.tails} 
Assume the conditions of Theorem \ref{t:weighted.sums}  
 but the vectors $\bfd_i$, $i=1,2,\ldots,$ may also  
contain zero components. Then,  
 for any  
Borel set $A\subset \bbr^d$ bounded away from the origin and \st\ $A$ is a 
$\mu_{\bfX}$-continuity set, 
\beam\label{eq:14aa} 
P(s^{-1}\bfX\in A)\sim \sum_{i=1}^\infty P\bigl( s^{-1}\Psi_i\bfZ\in A\bigr)\,,\quad  s\to \infty\,.  
\eeam   
\end{lemma} 
\begin{proof} 
For every $j=1,\ldots, d$, we may assume that there is 
$i(j)=1,2,\ldots$ such that $d^{(i(j))}_j\not =0$ for, if this is not 
the case, we can simply delete the $j$th coordinate. Denote 
$$ 
\bfY^{(j)} = \bfX-\Psi_{i(j)}\,\bfZ^{(i(j))} 
$$ 
and choose $M_j>0$ such that $P(\| \bfY^{(j)}\|\leq M_j)>0$, 
$j=1,\ldots, d$. We have for $s>0$ 
and $j=1,\ldots, d$,  
$$ 
P(\|\bfX\|>s) \geq P(\| \bfY^{(j)}\|\leq M_j) P\bigl(|d^{(i(j))}_j||Z_j|>s+M_j\bigr)\,, 
$$ 
 and the \regvar\ of $\bfX$ implies that there is $C_j>0$ \st 
\beao 
P(|Z_j|>s)\le C_j P(|\bfX\|>s)\,,\quad s>0\,, 
\eeao 
and therefore there is $C>0$ such
that  
 \beam\label{eq:15a} 
P(\|\bfZ\|>s)\le C\,P(\|\bfX\|>s)\quad \mbox{for all $s>0$}\,. 
\eeam 
 
We write $\bfX_q=\sum_{i=1}^q \Psi_i \bfZ^{(i)}$ and $\bfX^q=\bfX-\bfX_q$ 
for $q\ge 1$.  In the usual notation,  
$$ 
A^{\epsilon}=\{\bfy\in\bbr^d_0: d(\bfy,A)\le \epsilon\}, \ \  
A_{\epsilon}=\{\bfy\in A: d(\bfy,A^c)>\epsilon\}\,,  
$$ 
we have 
\beam\label{eq:w3a} 
\lefteqn{P(s^{-1}\bfX_q\in A_\epsilon)\,P(\|\bfX^q\|) 
\le \epsilon s)}\nonumber\\ 
&\le& P(s^{-1}\bfX\in A)\le P(s^{-1}\bfX_q\in A^\epsilon)+ 
P(\|\bfX^q\|> \epsilon s)\,. 
\eeam 
Proceeding as in the Appendix of 
\cite{mikosch:samorodnitsky:2000} and 
using \eqref{eq:15a}, we obtain 
\beao 
\lim_{q\to\infty}\limsup_{s \to \infty} 
\dfrac{P(\|\bfX^q\|>s)}{P(\|\bfX\|>s)} =0\,. 
\eeao 
Therefore and by virtue of \eqref{eq:w3a} it suffices to prove the 
lemma for $\bfX_q$ instead of $\bfX$. In what follows, we assume $q<\infty$ and 
suppress the dependence of $\bfX$ on $q$ in the notation.  
\par 
 
Let $M= \max_{i=1,\ldots,q, \, j=1,\ldots, d}|d^{(i)}_j| $. For   $\epsilon>0$ we have  
\beao 
\lefteqn{P(s^{-1}\bfX\in A_\epsilon)}\\ 
&\le &\sum_{j=1}^q P(s^{-1}\Psi_j\bfZ\in A)  
+ \dfrac{q\,(q-1)}{2} \left(P\left(\|\bfZ\|> 
    \dfrac{s\,\epsilon}{(q-1)M}\right) 
\right)^2\,.  
\eeao 
Hence, by \eqref{eq:15a} and \regvar\ of $\bfX$, 
\beao 
\mu_{\bfX}(A_\epsilon)\leq \liminf_{s\to \infty}\dfrac{P(s^{-1}\bfX\in 
  A_\epsilon)}{P(\|\bfX\|>s)} 
\leq \liminf_{s\to \infty}  \dfrac{\sum_{j=1}^q  P(s^{-1}\Psi_j\bfZ\in A)}{P(\|\bfX\|>s)} \,. 
\eeao 
Letting $\epsilon\downarrow 0$ and using that $A$ is a 
$\mu_{\bfX}$-continuity set, we have 
\beao 
\mu_{\bfX}(A)\le\liminf_{s\to \infty} 
\dfrac{ \sum_{j=1}^q   P(s^{-1}\Psi_j\bfZ\in A)}{P(\|\bfX\|>s)}\,,  
\eeao 
and \eqref{eq:14aa} will follow once we show that 
\beam\label{eq:17} 
\mu_{\bfX}(A)\ge \limsup_{s\to \infty}  
\dfrac{\sum_{j=1}^q  P(s^{-1}\Psi_j\,\bfZ\in A)}{P(\|\bfX\|>s)}\,. 
\eeam 
Let $\delta:= \inf\{ \|\bfx\|:\, \bfx\in A\}>0$. For $0<\epsilon<  
\delta$   write  
\beao 
P(s^{-1}\bfX\in A^\epsilon)&\ge &  
P\left( 
\bigcup_{i=1}^q 
\left\{ 
s^{-1}\Psi_i\bfZ^{(i)}\in 
  A\,,\left\|\sum_{1\le l\ne i\le q}\Psi_l\,\bfZ^{(l)}\right\|\le 
  s\epsilon 
\right\}\right)\\[2mm] 
&\ge &\sum_{i=1}^q P\left( 
s^{-1}\Psi_i\bfZ^{(i)}\in 
  A\,,\left\|\sum_{1\le l\ne i\le q}\Psi_l\,\bfZ^{(l)}\right\|\le 
  s\epsilon\right) \\[2mm] 
&&-\dfrac{q(q-1)}{2} \left(P(\|\bfZ\|\ge s\delta/M)\right)^2\\ 
&\ge &\sum_{i=1}^q P\left( 
s^{-1}\Psi_i\bfZ^{(i)}\in   A \right) 
-\dfrac{q(q-1)}{2} \left(P(\|\bfZ\|\ge s\delta/M)\right)^2 \\ 
&&- q(q-1) P(\|\bfZ\|\ge s\delta/M) P(\|\bfZ\|\ge s\epsilon/((q-1)M))\,. 
\eeao 
Thus  by \regvar\ of $\bfX$ and \eqref{eq:15a}, 
\beao 
\mu_{\bfX}(A^\epsilon)\geq \limsup_{s\to \infty} \dfrac{P(s^{-1}\bfX\in 
  A^\epsilon)}{P(\|\bfX\|>s)} 
\ge \limsup_{s\to \infty} \dfrac{\sum_{i=1}^q P\left( 
s^{-1}\Psi_j\bfZ\in 
  A\right)}{P(\|\bfX\|>s)}\,. 
\eeao 
Letting $\epsilon\downarrow 0$ and using the $\mu_{\bfX}$-continuity of 
$A$, we obtain the desired relation \eqref{eq:17}. 
\end{proof} 
\medskip 
 
\noindent 
\begin{proof}[Proof of Theorem \ref{t:weighted.sums}] 
It follows from Lemma \ref{l:sum.of.tails} that the measure  
$$ 
\mu(\cdot) = \sum_{i=1}^\infty P\bigl(  \Psi_i\bfZ\in \cdot\bigr) \ \ 
\text{on $\bbr^d$} 
$$ 
is regularly varying with index  $\alpha$. Note that $\mu=\nu 
\circledast \rho$, where $\nu$ is the law of $\bfZ$ (a probability 
measure), and 
$$ 
\rho = \sum_{i=1}^\infty \delta_{\bfd^{(i)}}\,, 
$$ 
with the usual notation $\delta_{\bfa}$ standing for the unit mass at 
the point $\bfa\in\bbr^d$. Note that the conditions of Theorem 
\ref{t:inverse.general} are satisfied; in particular 
\eqref{e:small.tail} holds because the measure $\rho$ has bounded 
support. Therefore, the conclusion of Theorem \ref{t:weighted.sums} 
follows.  
\end{proof} 
\bexam {\red 
Consider the vector AR(1) difference
equation $\bfX_t= \Psi \bfX_{t-1}+\bfZ_t$, $i\in\bbz$, for an iid
$\bbr^d$-valued \seq\ $(\bfZ_t)$ and a matrix $\Psi={\rm
diag}(\bfd)$ for some deterministic vector
$\bfd\in\bbd^d$ with nonvanishing coordinates. A unique stationary causal solution to the AR(1)
equation exists \fif\ $\max_{i=1,\ldots,d}|d_i|<1$ and $\bfZ_1$ has
some finite logarithmic moment. A generic 
element $\bfX$ of the solution satisfies the relation $\bfX\eqd
\sum_{j=0}^\infty \Psi^j\bfZ_j$.  Assume that $\bfX$ is \regvary\ 
with index $\alpha>0$.
Then  \eqref{e:small.d} is
trivially satisfied and \eqref{e:good.d} reads as follows: for every 
$j=1,\ldots,d$, any $m_i\in \{0,1\}, \theta_i\in\bbr$, $i=1,\ldots,d$,
\beao
|d_j|^\alpha \prod_{k=1}^d |d_k|^{i\theta_k}
\prod_{k=1}^d (\sign(d_k))^{m_k}
\Big(1-|d_j|^\alpha \prod_{k=1}^d |d_k|^{i\theta_k}
\prod_{k=1}^d (\sign(d_k))^{m_k}\Big)^{-1}\ne 0\,,
\eeao
This condition is always satisfied. Hence any $\bfZ_t$ is \regvary\
with index $\alpha>0$.
}
\eexam

A special case of the setup of this section is a sum with scalar 
weights, of the type   
$\bfX=\sum_{i=1}^\infty \psi_i \bfZ^{(i)}$, where $(\psi_i)$ is a 
sequence of scalars. Applying  Theorem \ref{t:weighted.sums} with  
$d_j^{(i)}=\psi_i, \, j=1,\ldots, d$ for $i=1,2,\ldots$, we obtain the 
following corollary.  
\begin{corollary}\label{prop:3} 
Let $\alpha>0$, and suppose that for some $0<\delta<\alpha$, 
\begin{equation} \label{e:small.psi} 
 \sum_{i=1}^\infty|\psi_i|^{\alpha-\delta}<\infty\,. 
\end{equation}  
 Assume that the series $\bfX = \sum_{i=1}^\infty \psi_i\,\bfZ^{(i)}$ 
converges a.s. and $\bfX$ is \regvary\ 
with index $\alpha$. If  
\beam\label{eq:13a} 
\sum_{j=1}^\infty |\psi_j|^{\alpha+i\theta}&\ne& 0\,,\quad \theta\in 
\bbr\,,\quad  \mbox{and} \\[2mm] 
\sum_{j: \psi_j>0}\psi_j^{\alpha+i\theta}&\ne&  
\sum_{j: \psi_j<0}|\psi_j|^{\alpha+i\theta}\,,\quad \theta\in 
\bbr\,,\label{eq:13b} 
\eeam 
 then $\bfZ\in\RV(\alpha,\mu_{\bfZ})$ and the tail 
 measure $\mu_\bfZ$ satisfies  
\beao 
\dfrac{P(s^{-1}\bfX \in \cdot)}{P(|\bfZ|>s)}\stv  
  \psi_+ \,\mu_\bfZ(\cdot) + 
 \psi_- \,\mu_\bfZ(-\cdot)\,,\quad s\to\infty\,, 
\eeao 
where 
\beam\label{eq:psi} 
\psi_+=\sum_{j:\psi_j>0} \psi_j^\alpha\quad \mbox{and}\quad  
\psi_-=\sum_{j:\psi_j<0} |\psi_j|^\alpha\,. 
\eeam 
\end{corollary} 
 
\begin{remark} \label{rk:converse} 
{\rm  
Corollary \ref{prop:3} has a natural converse statement. Specifically, 
if either \eqref{eq:13a} or \eqref{eq:13b} fail to hold for some real 
$\theta$, then there is a random vector $\bfZ$ that is not regularly 
varying but $\bfX=\sum_{i=1}^\infty \psi_i \bfZ^{(i)}$ is 
regularly varying. Indeed, suppose, for example, that \eqref{eq:13a} 
fails for some real $\theta_0$. We use a construction similar to that in  
\cite{jacobsen:mikosch:rosinski:samorodnitsky:2008}.  
Choose real numbers $a,b$ satisfying 
$0<a^2+b^2\le 1$, and define a measure on $(0,\infty)$ by 
\begin{equation} \label{e:nu.0} 
\nu_0(dx)=[1+ a\,\cos(\theta_0\,\log x)+b\,\sin(\theta_0\,\log 
x)]\, 
\nu^\alpha(dx)\,,  
\end{equation}   
where $\nu^\alpha$ is given in \eqref{e:nu.alpha}. Choose $r>0$ large 
enough so  
that $\nu_0(r,\infty)\leq 1$, define a probability law 
on $(0,\infty)$ by  
$$ 
\mu_0(B) = \nu_0\bigl( B\cap (r,\infty)\bigr) + \bigl[ 
  1-\nu_0(r,\infty)\bigr]\bfe_B(1)\quad\mbox{for any Borel set $B$,} 
$$  
and a probability law on $\bbr$ by  
$$ 
\mu_\ast(\cdot) = \frac12 \mu_0(\cdot) + \frac12 \mu_0(-\cdot). 
$$ 
Obviously, $\mu_\ast$ is not a regularly varying probability 
measure. Therefore, neither is the random vector $\bfZ=(Z,0, \ldots, 
0)$ regularly varying, where $Z$ has \ds\ $\mu_\ast$.  
 
Since the vector $\bfZ$ is symmetric, the series 
$\bfX=\sum_{i=1}^\infty \psi_i \bfZ^{(i)}$ converges a.s. under the 
assumption \eqref{e:small.psi}; see Lemma A.3 in 
\cite{mikosch:samorodnitsky:2000}, and the argument in 
\cite{jacobsen:mikosch:rosinski:samorodnitsky:2008} shows that $\bfX$ 
is regularly varying with index $\alpha$.  
 
On the other hand, suppose  that \eqref{eq:13b} fails for some real 
$\theta_0$. Define $\nu_0$ as in \eqref {e:nu.0}, and define another measure 
on $(0,\infty)$ by 
$$ 
\nu_1(dx)=[1-a\,\cos(\theta_0\,\log x)-b\,\sin(\theta_0\,\log 
x)]\, 
\nu^\alpha(dx)\,. 
$$ 
Convert $\nu_0$ into a probability measure $\mu_0$ as above, and 
similarly convert $\nu_1$ into a probability measure $\mu_1$. Define a 
probability measure on $\bbr$ by  
$$ 
\mu_\ast(\cdot) = \frac12 \mu_0(\cdot) + \frac12 \mu_1(-\cdot). 
$$ 
Once again, let $\bfZ=(Z_1,0, \ldots, 
0)$, where $Z_1\sim \mu_\ast$. Then $\bfZ$ is not regularly 
varying, and neither is the vector  
$$ 
\tilde \bfZ = \left\{ 
\begin{array}{ll} 
\bfZ\,,  & \text{if} \ 0<\alpha\leq 1\,,\\ 
\bfZ-E(\bfZ)\,, & \text{if} \ \alpha> 1\,. 
\end{array} 
\right. 
$$ 
As before, the series 
$\bfX=\sum_{i=1}^\infty \psi_i \tilde \bfZ^{(i)}$ converges a.s. under the 
assumption \eqref{e:small.psi}, and $\bfX$ 
is regularly varying with index  $\alpha$.  
} 
\end{remark} 
We proceed with several examples of the situation described in 
Corollary \ref{prop:3}. We say that the coefficients $\psi_1,\psi_2,\ldots,$ 
are {\em $\alpha$-\regvar\ determining} if \regvar\ of 
$\bfX=\sum_{i=1}^\infty \psi_i \bfZ^{(i)}$ implies \regvar\ 
of $\bfZ$. In other words,  both conditions \eqref{eq:13a} and 
\eqref{eq:13b} 
must be satisfied. 
\bexam\label{exam:1} 
Let $q<\infty$ and assume that $\psi_i=1$, $i=1,\ldots,q$,  $\psi_i=0$ 
for $i>q$. By Corollary \ref{prop:3} 
these coefficients are $\alpha$-\regvar\ determining 
and $P(s^{-1}\bfX \in \cdot) \sim q P(s^{-1}\bfZ \in \cdot)$ as 
$s\to\infty$.  For $d=1$ (only in this case the notion of 
subexponentiality is properly defined) this property is in agreement 
with the   
{\em convolution root property} of subexponential \ds s; see   
\cite{embrechts:goldie:veraverbeke:1979};  
cf. Proposition~A3.18 in  
al. \cite{embrechts:kluppelberg:mikosch:1997}. Indeed, if $X$ is a  
positive \rv\ then \regvar\ of $Z$ implies subexponentiality. 
\eexam 
\bexam\label{exam:2} 
Again, let $q<\infty$ and  $\psi_j=0$ for $j>q$.  
If, say, $|\psi_1|^\alpha>\sum_{j=2}^q |\psi_j|^\alpha$, then  
both conditions 
\eqref{eq:13a} and \eqref{eq:13b} are satisfied and therefore 
 the coefficients are $\alpha$-\regvar\ determining. This is, of 
course, the same phenomenon as in  Corollary 
\ref{c:atoms.cancel}. In the special case, $q=2$,   
if $\psi_1\not=-\psi_2$,   then the coefficients are $\alpha$-\regvar\ 
determining. On the other hand,  If $\psi_1=-\psi_2$,   then 
condition \eqref{eq:13b} fails, and the coefficients are not 
$\alpha$-\regvar\ determining.  
This means that \regvar\  of 
$\bfX=\bfZ_1-\bfZ_2$ does not necessarily imply \regvar\ of $\bfZ$.  
\eexam 
 

\section{The inverse problem for products} \label{sec:products} 
\setcounter{equation}{0} 
 
We now apply the generalized cancellation theory to Example 
\ref{ex:products} above. We concentrate on the case of 
multiplication by a random diagonal matrix. Specifically, let $\bfA= 
{\rm diag}(A_1,\ldots,A_d)$ for some \rv s $(A_i)$, 
$i=1,\ldots,d$. The following theorem is an easy application of 
Theorem \ref{t:inverse.general}.  
\begin{theorem} \label{t:product.inverse} 
Assume that $P(A_j=0)=0$ for $j=1,\ldots, d$. Let $\bfA= 
{\rm diag}(A_1,\ldots,A_d)$. Let $\bfZ$ be a $d$-dimensional random 
vector independent of $\bfA$, such that $\bfX=\bfA \bfZ$ is regularly 
varying with index $\alpha>0$. If $E\|\bfA\|^{\alpha+\delta}< 
\infty$ for some $\delta>0$ and  
\begin{equation} \label{e:cond.product} 
E\left( |A_j|^\alpha \prod_{k=1}^d \bigl( |A_k|^{i\theta_k} ({\rm 
    sign}(A_k))^{m_k}\bigr)\right) \not= 0 
\end{equation} 
for each $j=1,\ldots, d$, $m_1,\ldots, m_d\in \{0,1\}$ and 
$\theta_1,\ldots, \theta_d\in \bbr$, then $\bfZ$ is regularly 
varying with index $\alpha>0$. Moreover, \eqref{e:product.regvar} 
holds.  
\end{theorem} 
A special case is multiplication 
of a random vector by an independent scalar random variable, 
corresponding to $A_1=\ldots = A_d=A$ for 
some random variable $A$. The following corollary restates Theorem 
\ref{t:product.inverse} in this special case.  
\begin{corollary} \label{c:1dimprod.inverse} 
Let $A$ be a random variable independent of a $d$-dimensional random 
vector $\bfZ$ such that $\bfX=A \bfZ$ is regularly 
varying with index $\alpha>0$. If $E|A|^{\alpha+\delta}< 
\infty$ for some $\delta>0$ and  
\beam\label{eq:sym} 
&&E|A|^{\alpha+i\theta}\ne 0\,,\quad \theta\in\bbr\,,\\ 
&&EA_+^{\alpha+i\theta}\ne EA_-^{\alpha+i\theta}\,,\quad \theta\in\bbr\,,\label{eq:syma} 
\eeam 
then $\bfZ$ is regularly 
varying with index $\alpha>0$. Moreover,  the tail 
 measure $\mu_\bfZ$ of $\bfZ$  satisfies  
\beao 
\dfrac{P(s^{-1}\bfX \in \cdot)}{P(|\bfZ|>s)}\stv  
   EA_+^{\alpha}\,\mu_\bfZ(\cdot) + 
  EA_-^{\alpha} \,\mu_\bfZ(-\cdot)\,,\quad s\to\infty\,, 
\eeao 
where $A_+=\max(A,0), \, A_-=\max(-A,0)$.  
\end{corollary} 
Using terminology similar to that of the previous section, we say that a  
random variable $A$ is  
{\em $\alpha$-\regvar\ determining} if \regvar\  
 with index $\alpha$ of $\bfX=A \bfZ$  
 implies \regvar\ 
of $\bfZ$.  Corollary \ref{c:1dimprod.inverse} shows that if $A$ 
satisfies both conditions \eqref{eq:sym} and \eqref{eq:syma}, then $A$ 
is  $\alpha$-\regvar\ determining. On the other hand,  
a construction similar to that in Remark 
\ref{rk:converse} shows that, if one of the conditions \eqref{eq:sym} 
and \eqref{eq:syma} fails, then one can construct an example of a 
random vector $\bfZ$ that is not regularly varying but $\bfX=A 
\bfZ$ is regularly varying with index $\alpha$. Therefore,  
conditions \eqref{eq:sym} and \eqref{eq:syma} are necessary and 
sufficient for $A$ being $\alpha$-\regvar\ determining.  
 
\cite{jacobsen:mikosch:rosinski:samorodnitsky:2008}, 
Theorem~4.2,  
proved this result for positive $A$. They gave various examples 
of \ds s on $(0,\infty)$ which are $\alpha$-\regvar\ determining, 
including the gamma, log-normal, Pareto \ds s, the \ds\ of the  
powers of the absolute value of a symmetric normal random variable,  
 of the   
absolute values of a Cauchy random variable  
(for $\alpha<1$) and any positive \rv\ whose  
log-transform is \id .   The condition in \eqref{eq:syma} rules  
out a whole 
class of important \ds s: no member of the class of symmetric \ds s is 
$\alpha$-\regvar\ determining. Even  
non-symmetric \ds s with $EA_+^\alpha=EA_-^{\alpha}$ are not 
$\alpha$-\regvar\ determining.  For a further example, consider a 
uniform random variable $A\sim U(a,b)$ for $a<b$. If $a=-b$, then $A$ 
cannot be $\alpha$-\regvar\ determining since it has a symmetric 
distribution. On the other hand, an elementary calculation shows that 
in all other cases both conditions \eqref{eq:sym} 
and \eqref{eq:syma} hold. Therefore, the only non-$\alpha$-\regvar\ 
determining uniform random variables are the symmetric ones.   
 
In financial \tsa , models for returns are often  
of the form $X_t=A_tZ_t$, where $(A_t)$ is some volatility \seq\ and 
$(Z_t)$ is an iid multiplicative noise \seq\ \st\ $A_t$ and $Z_t$ are 
independent for every $t$ and $(X_t)$ constitutes a strictly 
stationary \seq . In most parts of the literature it is assumed that 
the volatility $A_t$ is non-negative. It is often assumed that  
$X_t$ is heavy-tailed, e.g. \regvary\ with some index 
$\alpha>0$; see \cite{davis:mikosch:2009a,davis:mikosch:2009b}. 
 Notice that $A_t$ and $Z_t$ are not observable;  it depends 
on the model to which of the variables $A_t$ or $Z_t$ 
one assigns \regvar . For example, in the case of a GARCH process $(X_t)$, 
$(A_t)$ is \regvary\ with index $\alpha>0$ and the iid noise $(Z_t)$ 
has  lighter tails and  is symmetric.  
On the other hand, if one only assumes that $X_t$ is \regvary\ with 
index $\alpha$  
and  $E|Z|^{\alpha+\delta}<\infty$ for some $\delta>0$ and  
$Z$  is symmetric,  one cannot conclude that $A_t$ is \regvary .

\section{Non-diagonal matrices} \label{sec:non-diagonal} 
\setcounter{equation}{0} 
 
The (direct) statements of Examples \ref{ex:linear.process} and 
\ref{ex:products} of Section \ref{sec:1}  
deal with transformations of regularly varying 
random vectors involving matrices that do not have to be diagonal 
matrices. On the other hand, all the converse statements of Sections 
\ref{sec:sums} and \ref{sec:products} deal only with diagonal 
matrices.  Generally, we do not know how to solve inverse problems 
involving non-diagonal matrices. This section describes one of the 
very few ``non-diagonal'' situations where we can prove a converse 
statement.  We restrict ourselves to the case of finite weighted 
sums and square matrices.  
 
\begin{theorem} \label{prop:5} 
Let  
$\bfX=\sum_{j=1}^q \Psi_j\,\bfZ_j$, where  
$\bfZ_j$, $j=1,\ldots,q$, are iid $\bbr^d$-valued random vectors and  
$\Psi_j$, $j=1,\ldots,q$,  deterministic $(d\times 
d)$-matrices. Assume that $\bfX\in\RV(\alpha,\mu_{\bfX})$ for some 
$\alpha>0$.  If all the matrices $\Psi_j$, $j=1,\ldots,q$ are 
invertible, and  
\beam\label{eq:alph} 
(\gamma(\Psi_1))^{\alpha} >\sum_{j=2}^q \|\Psi_j\|^{\alpha}\,, 
\eeam 
where $\gamma(\Psi_1)=\min_{\bfz\in \bbs^{d-1}} |\Psi_1\bfz|$ 
and $\|\Psi_j\|$ is the operator norm of $\Psi_j$, $j=1,\ldots,q$,  
then $\bfZ\in {\rm RV}(\alpha,\mu_\bfZ)$ and $\mu_{\bfZ}$ satisfies  
\eqref{eq:10}.  
\end{theorem} 
\begin{proof} 
An argument similar to that in Lemma \ref{l:sum.of.tails} shows that 
under the assumptions of the theorem a finite version of 
\eqref{eq:14aa} holds: for any  
Borel set $A\subset \bbr^d$ bounded away from the origin \st\ $A$ is a 
continuity set with respect to the tail measure $\mu_{\bfX}$, 
$$ 
P(s^{-1}\bfX\in A)\sim \sum_{i=1}^q P\bigl( s^{-1}\Psi_i\bfZ\in 
A\bigr)\,,\quad s\to \infty\,.   
$$ 
This allows us to proceed as in Theorem \ref{t:subseq.char} to see 
that the family of measures  
\begin{equation} \label{e:an.fam} 
\left( \dfrac{P(s^{-1}\bfZ\in \cdot )}{P(|\bfX|>s)}\right)_{s \ge 1} 
\end{equation} 
is vaguely tight in $\ov \bbr^m_0$, and any vague subsequent limit 
$\mu_\ast$ of this family satisfies  
\beam\label{eq:lat} 
\mu_\bfX =\sum_{j=1}^q 
\mu_{\ast}\circ \Psi_j^{-1}\,. 
\eeam 
Let $T_j= \Psi_j^{-1}\Psi_1$, $j=2,\ldots, q$. Then by \eqref{eq:lat},   
for any measurable set $B\subset \bbr^d$ bounded away from zero,   
\beam\label{eq:iter} 
\mu_\ast(B)=\mu_\bfX(\Psi_1 B)- \sum_{j=2}^q \mu_\ast(T_j B)\,. 
\eeam 
Replacing $B$ with $T_j B$ for $j=2,\ldots, q$ and iterating 
\eqref{eq:iter}, we obtain for $n=1,2,\ldots,$ 
\beam\label{eq:nu} 
\mu_\ast(B)&=&\mu_\bfX(\Psi_1 B)- \sum_{j=2}^q \mu_\bfX(\Psi_1 T_j B) 
+ \sum_{j_1=2}^q\sum_{j_2=2}^q 
\mu_\ast(T_{j_2} T_{j_1} B)\nonumber\\ &=&\cdots\nonumber  \\[2mm] 
&=&\sum_{k=0}^n (-1)^k\sum_{j_1=2}^q \cdots\sum_{j_k=2}^q \mu_\bfX 
(\Psi_1T_{j_k}\cdots 
T_{j_1} B)\\ 
&&+(-1)^{n+1}\sum_{j_1=2}^q \cdots \sum_{j_{n+1}=2}^q 
\mu_\ast(T_{j_{n+1}}\cdots 
T_{j_1} B)\,.\nonumber 
\eeam 
Clearly,  for every $n\ge 1$ and $j_1,\ldots,j_{n+1} =2,\ldots,q$, 
\beam\label{eq:con} 
\inf_{\bfz\in T_{j_{n+1}}\cdots T_{j_1}B} |\bfz|\ge  
\inf_{\bfz\in B} |\bfz|\,(\gamma(\Psi_1))^{n+1}\,\prod_{k=1}^{n+1} \|\Psi_{j_k}\|^{-1}\,. 
\eeam 
Furthermore, it follows from  \eqref{eq:lat} that, for some 
$c>0$,  
\beao 
\mu_\ast(\{\bfz\in\bbr^d: |\bfz|>s\})\le c\, s^{-\alpha}\,,\quad s>0\,. 
\eeao 
Therefore we conclude by \eqref{eq:con} and  \eqref{eq:alph} that 
\beao 
\lefteqn{\sum_{j_1=2}^q \cdots \sum_{j_{n+1}=2}^q 
\mu_\ast(T_{j_{n+1}}\cdots 
T_{j_1} B)}\\[2mm]&\le& c \,\bigl( \inf_{\bfz\in B} |\bfz|\bigr)^{-\alpha}  
\sum_{j_1=2}^q \cdots \sum_{j_{n+1}=2}^q 
\left(\,(\gamma(\Psi_1))^{n+1}\,\prod_{k=1}^{n+1} 
  \|\Psi_{j_k}\|^{-1}\right)^{-\alpha}\\[2mm] 
&=&c\,\, \bigl( \inf_{\bfz\in B} |\bfz|\bigr)^{-\alpha}  
(\gamma(\Psi_1))^{-\alpha \,(n+1)} 
\left(\sum_{j=2}^q 
  \|\Psi_{j}\|^{\alpha}\right)^{n+1}\to 0\,. 
\eeao 
Thus by virtue of \eqref{eq:nu}, 
\beao 
\mu_\ast(B)=\lim_{\nto}\sum_{k=0}^n (-1)^k\sum_{j_1=2}^q \cdots\sum_{j_k=2}^q \mu_\bfX 
(T_{j_k}\cdots T_{j_1} B)\,. 
\eeao 
This means that $\mu_\ast$ is uniquely determined by the \ms\ 
$\mu_\bfX$. 
Hence all the subsequential vague limits of \eqref{e:an.fam} coincide. 
Therefore, $\bfZ\in {\rm RV}(\alpha,\mu_\bfZ)$ and \eqref{eq:10} holds.  
\end{proof} 
\begin{remark} \label{rk:reduce} 
{ 
\rm 
Note that in the special case of diagonal matrices $(\Psi_j)$ with 
identical elements on the diagonals, the conditions in Theorem \ref{prop:5} 
coincide with those in Example \ref{exam:2} above.  
} 
\end{remark} 
 
\bre\label{rem:1}{\em  
The conditions in Theorem \ref{prop:5} can be slightly weakened by 
assuming, instead of  \eqref{eq:alph}, that 
\beam \label{e:mult.A} 
(\gamma(A\Psi_1))^{\alpha} >\sum_{j=2}^q \|A\Psi_j\|^{\alpha} 
\eeam  
for some invertible matrix $A$. Indeed,  regular variation of $\bfX$ 
implies regular variation of the  vector $A\bfX$, and regular 
variation of the  $A\bfZ$ implies regular variation of $\bfZ$. It is 
not difficult to construct examples where \eqref{e:mult.A} holds but 
\eqref{eq:alph} fails.}  
\ere 
  
\noindent 
{\bf Acknowledgment.} Ewa Damek's research  
was partially supported by  
the NCN grant DEC-2012/05/B/ST1/00692. 
Thomas Mikosch's  research  
was partially supported by  
the Danish Natural Science Research Council (FNU) Grants 
09-072331 "Point process modelling and statistical inference" 
and 10-084172 ``Heavy tail phenomena: Modeling and estimation''. 
Jan Rosi\'nski's research  
was partially supported by the Simons Foundation grant 281440.  
Gennady Samorodnitsky's research was 
partially supported by the NSF grant DMS-1005903,  
 and the ARO grants W911NF-07-1-0078 and and W911NF-12-10385 at 
 Cornell University. All authors are very much indebted to Darek 
 Buraczewski (University of Wroclaw) who suggested to use Theorem 
9.3 in \cite{rudin:1973} instead of the theory of distributions which was 
employed in \cite{jacobsen:mikosch:rosinski:samorodnitsky:2008}. 
This approach simplified our proofs.

\end{document}